\documentclass[reqno]{amsart}
\usepackage{mathrsfs}
\usepackage{color}
\usepackage{amsmath}
\usepackage{amsfonts}
\usepackage{amssymb}
\usepackage{graphicx}

 \newtheorem{Theorem}{Theorem}[section]
 \newtheorem{Corollary}[Theorem]{Corollary}
 \newtheorem{Lemma}[Theorem]{Lemma}

 \newtheorem{Remark}[Theorem]{Remark}

 \numberwithin{equation}{section}

\begin{document}

\title{Modules at boundary points, fiberwise Bergman kernels, and log-subharmonicity}
\author{Shijie Bao}
\address{Shijie Bao: Institute of Mathematics, Academy of Mathematics
and Systems Science, Chinese Academy of Sciences, Beijing 100190, China.}
\email{bsjie@amss.ac.cn}

\author{Qi'an Guan}
\address{Qi'an Guan: School of
Mathematical Sciences, Peking University, Beijing 100871, China.}
\email{guanqian@math.pku.edu.cn}


\thanks{}


\date{\today}

\dedicatory{}

\commby{}

\maketitle
\begin{abstract}
In this article, we consider Bergman kernels with respect to modules at boundary points, and obtain a log-subharmonicity property of the Bergman kernels,
which deduces a concavity property related to the Bergman kernels.
As applications, we reprove the sharp effectiveness result related to a conjecture posed by Jonsson-Must\c{a}t\u{a} and the effectiveness result of strong openness property of the modules at boundary points. 
\end{abstract}

\section{Introduction}

The strong openness property of multiplier ideal sheaves (i.e. $\mathcal{I}(\psi)=\mathcal{I}_+(\psi):=\mathop{\cup} \limits_{\epsilon>0}\mathcal{I}((1+\epsilon)\psi)$) is an important feature
 and has been widely used in the study of several complex variables, complex algebraic geometry and complex differential geometry
(see e.g. \cite{GZSOC,K16,cao17,cdM17,FoW18,DEL18,ZZ2018,GZ20,berndtsson20,ZZ2019,ZhouZhu20siu's,FoW20,KS20,DEL21}),
where $\psi$ is a plurisubharmonic function on a complex manifold $M$ (see \cite{Demaillybook}) and multiplier ideal sheaf $\mathcal{I}(\psi)$ is the sheaf of germs of holomorphic functions $f$ such that $|f|^2e^{-\psi}$ is locally integrable (see e.g. \cite{Tian,Nadel,Siu96,DEL,DK01,DemaillySoc,DP03,Lazarsfeld,Siu05,Siu09,DemaillyAG,Guenancia}).

The strong openness property was conjectured by Demailly \cite{DemaillySoc} and proved by Guan-Zhou \cite{GZSOC} (the 2-dimensional case was proved by Jonsson-Musta\c{t}\u{a}
\cite{JM12}). Recall that in order to prove the strong openness property, Jonsson and Musta\c{t}\u{a} (see \cite{JM13}, see also \cite{JM12}) posed the following conjecture, and proved the 2-dimensional case \cite{JM12}:

\textbf{Conjecture J-M}: If $c_o^F(\psi)<+\infty$, $\frac{1}{r^2}\mu(\{c_o^F(\psi)\psi-\log|F|<\log r\})$ has a uniform positive lower bound independent of $r\in(0,1)$, 
where $\mu$ is the Lebesgue measure on $\mathbb{C}^n$, and $c_o^{F}(\psi):=\sup\{c\geq 0 : |F|^2e^{-2c\psi}$ is locally $L^1$ near $o \}$.

Using the strong openness property, Guan-Zhou \cite{GZeff} proved Conjecture J-M.

Independent of the strong openness property, Bao-Guan-Yuan \cite{BGY} considered minimal $L^{2}$ integrals with respect to a module at a boundary point of the sublevel sets, 
and established a concavity property of the minimal $L^{2}$ integrals, which deduced a sharp effectiveness result related to Conjecture J-M,
and completed the approach from Conjecture J-M to the strong openness property.

As a generalization of Berndtsson's log-plurisubharmonicity result of fiberwise Bergman kernels (see \cite{Blogsub}), in \cite{BG1} (see also \cite{BG2}), we obtained the log-plurisubharmonicity of fiberwise Bergman kernels with respect to functionals over the space of holomorphic germs by using the optimal $L^2$ extension theorem (see \cite{guan-zhou13ap}) and Guan-Zhou method (see \cite{Ohsawa}). As applications, we gave new approaches to the effectiveness results of strong openness property (\cite{BG1}) and $L^p$ strong openness property (\cite{BG2}).  

As continuity work of \cite{BG1} and \cite{BG2}, in this article, we consider Bergman kernels with respect to the modules at boundary points, and obtain a log-subharmonicity property of the Bergman kernels (as a generalization of Berndtsson's log-subharmonicity result of fiberwise Bergman kernels in \cite{Blogsub}), which deduces a new approach from Conjecture J-M to the strong openness property. We also give a reproof for the effectiveness result of the strong openness property related to the modules at boundary points.

\subsection{Main result}\label{main}
\

Let $D$ be a pseudoconvex domain in $\mathbb{C}^n$, and the origin $o\in D$. Let $F\not\equiv 0$ be a holomorphic function on $D$, and $\psi$ be a negative plurisubharmonic function on $D$. Let $\varphi_0$ be a plurisubharmonic function on $D$. Denote that
\[\Psi:=\min\{\psi-2\log|F|, 0\}.\]
If $F(w)=0$ for $w\in D$, set $\Psi(w)=0$.

We recall some notations in \cite{BGY}. Denote that
\[\tilde{J}(\Psi)_o:=\{f\in\mathcal{O}(\{\Psi<-t\}\cap V) : t\in\mathbb{R},\ V \text{\ is \ a \ neighborhood \ of \ } o\},\]
and
\[J(\Psi)_o:=\tilde{J}(\Psi)_o/\sim,\]
where the equivalence relation `$\sim$' is as follows:
\[f \sim g \ \Leftrightarrow \ f=g \text{\ on \ } \{\Psi<-t\}\cap V, \text{\ where\ }  t\gg 1, V \text{\ is\ a\ neighborhood\ of\ } o.\]
For any $f\in \tilde{J}(\Psi)_o$, denote the equivalence class of $f$ in $J(\Psi)_o$ by $f_o$. And for any $f_o,g_o\in J(\Psi)_o$, and $(h,o)\in\mathcal{O}_o$, define
\[f_o+g_o:=(f+g)_o,\ (h,o)\cdot f_o:=(hf)_o.\]
It is clear that $J(\Psi)_o$ is an $\mathcal{O}_{o}-$module. For any $a\geq 0$, denote that $I(a\Psi+\varphi_0)_o:=\big\{f_o\in J(\Psi)_o : \exists t\gg 1, V \text{\ is\ a\ neighborhood\ of\ } o,\ \text{s.t.\ } \int_{\{\Psi<-t\}\cap V}|f|^2e^{-a\Psi-\varphi_0}<+\infty\big\}.$ Then it is clear that $I(a\Psi+\varphi_0)_o$ is an $\mathcal{O}_{o}-$submodule of $J(\Psi)_o$. Especially, we denote that $I(\varphi_0)_o:=I(0\Psi+\varphi_0)_o$, $I(\Psi)_o:=I(\Psi+0)_o$, and $I_o:=I(0\Psi+0)_o$. Then $I(a\Psi+\varphi_0)_o$ is an $\mathcal{O}_{o}-$submodule of $I(\varphi_0)_o$ for any $a>0$.

 For any $t\in [0,+\infty)$ and $\lambda>0$, denote that
\[\Psi_{\lambda,t}:=\lambda\max\{\Psi+t,0\},\]
and for any $f\in A^2(\{\Psi<0\},e^{-\varphi_0})$ and $\lambda>0$, denote that
\[\|f\|_{\lambda,t}:=\left(\int_{\{\Psi<0\}}|f|^2e^{-\varphi_0-\Psi_{\lambda,t}}\right)^{1/2},\]
where $A^2(\{\Psi<0\},e^{-\varphi_0}):=\{f\in\mathcal{O}(\{\Psi<0\}) : \int_{\{\Psi<0\}}|f|^2e^{-\varphi_0}<+\infty\}$ (if $\varphi_0\equiv 0$, we may denote $A^2(\{\Psi<0\}):=A^2(\{\Psi<0\},e^0)$). It is clear that $e^{-\lambda t/2}\|f\|_{\lambda,0}\leq\|f\|_{\lambda,t}\leq\|f\|_{\lambda,0}<+\infty$ for any $t\geq 0$.

For any $\xi\in A^2(\{\Psi<0\},e^{-\varphi_0})^*$ (the dual space of $A^2(\{\Psi<0\},e^{-\varphi_0})$), denote that the Bergman kernel related to $\xi$ is
\[K^{\varphi_0}_{\xi,\Psi,\lambda}(t):=\sup_{f\in A^2(\{\Psi<0\},e^{-\varphi_0})}\frac{|\xi\cdot f|^2}{\|f\|_{\lambda,t}^2}\]
for any $t\in [0,+\infty)$.

Denote $E:=\{w\in\mathbb{C} :\text{Re\ } w\geq 0\}\subset\mathbb{C}$. 
We obtain the following log-subharmonicity property of the Bergman kernel $K^{\varphi_0}_{\xi,\Psi,\lambda}$.

\begin{Theorem}\label{concavity}
	Assume that $K^{\varphi_0}_{\xi,\Psi,\lambda}(0)\in (0,+\infty)$. Then $\log K^{\varphi_0}_{\xi,\Psi,\lambda}(\text{Re\ } w)$ is subharmonic with respect to $w\in E$.
\end{Theorem}

Let $J$ be an $\mathcal{O}_{o}-$submodule of $I(\varphi_0)_o$. Denote that
\[A^2(\{\Psi<0\},e^{-\varphi_0})\cap J:=\{f\in A^2(\{\Psi<0\},e^{-\varphi_0}) : f_o\in J\}.\]
Assume that $A^2(\{\Psi<0\},e^{-\varphi_0})\cap J$ is a proper subspace of $A^2(\{\Psi<0\},e^{-\varphi_0})$ (and we will state that it is a closed subspace). 

Using Theorem \ref{concavity}, we obtain the following concavity and monotonicity property related to $K^{\varphi_0}_{\xi,\Psi,\lambda}$.

\begin{Theorem}\label{increasing}
	Assume that $J\supset I(\Psi+\varphi_0)_o$, and assume that $\xi\in A^2(\{\Psi<0\},e^{-\varphi_0})^*$ such that $\xi|_{A^2(\{\Psi<0\},e^{-\varphi_0})\cap J }\equiv 0$ and $K^{\varphi_0}_{\xi,\Psi,\lambda}(0)\in (0,+\infty)$. Then $-\log K^{\varphi_0}_{\xi,\Psi,\lambda}(t)+t$ is concave and increasing with respect to $t\in [0,+\infty)$.
\end{Theorem}

\begin{Remark}\label{notconstant}
	Let $\xi\in A^2(\{\Psi<0\},e^{-\varphi_0})^*$. According to Theorem \ref{concavity}, if there exist $k>0$ and $T>0$, such that $e^{-kt} K^{\varphi_0}_{\xi,\Psi,\lambda}(t)$ is increasing and not a constant function on $[0,T]$, then $e^{-kt}K^{\varphi_0}_{\xi,\Psi,\lambda}(t)$ is strictly increasing on $[T,+\infty)$.
\end{Remark}

\subsection{Applications}
\

As applications of Theorem \ref{concavity} and Theorem \ref{increasing}, we give new proofs of some results in \cite{BGY,GMY-BC2}.

Let $D$ be a pseudoconvex domain in $\mathbb{C}^n$, and the origin $o\in D$. Let $F\not\equiv 0$ be a holomorphic function on $D$, and $\psi$ be a negative plurisubharmonic function on $D$. Denote that
\[\Psi:=\min\{\psi-2\log|F|, 0\}.\]
If $F(w)=0$ for $w\in D$, set $\Psi(w)=0$.

Let $f$ be a holomorphic function on $D$. Recall the definition of the minimal $L^{2}$ integral related to $J$ (\cite{BGY,GMY-BC2})
\[G(t;\Psi,J,f):=\inf\left\{\int_{\{\Psi<-t\}}|\tilde{f}|^2 : \tilde{f}\in\mathcal{O}(\{\Psi<-t\}) \ \& \ (\tilde{f}-f)_o\in J\right\}\]
for any $\mathcal{O}_o-$submodule $J$ of $I_o$ and $t\in[0,+\infty)$. Denote that
\[\Psi_1:=\min\{2c_o^{fF}(\psi)\psi-2\log|F|,0\},\]
and
\[I_+(\Psi_1)_o:=\bigcup_{a>1}I(a\Psi_1)_o,\]
where $c_o^{fF}(\psi):=\sup\{c\geq 0 : |fF|^2e^{-2c\psi}$ is locally $L^1$ near $o \}$.

Theorem \ref{increasing} deduces a reproof of the following lower bound of $L^{2}$ integrals.
\begin{Corollary}[\cite{BGY}]\label{J-M+f}
	If $f\in A^2(\{\Psi_1<0\})$, and $c_o^{fF}(\psi)<+\infty$, then for any $r\in (0,1]$,
	\[\frac{1}{r^2}\int_{\{c_o^{fF}(\psi)\psi-\log|F|<\log r\}}|f|^2\geq G(0;\Psi_1,I_+(\Psi_1)_o,f)>0.\]
\end{Corollary}

\begin{Remark}
	The proof of the inequality $G(0;\Psi_1,I_+(\Psi_1)_o,f)>0$ can be referred to \cite{BGY}.
\end{Remark}

When $f\equiv 1$, Corollary \ref{J-M+f} deduces a reproof of the sharp effectiveness result related to a conjecture posed by Jonsson-Musta\c{t}\u{a}.
\begin{Corollary}[\cite{BGY}]\label{J-M}
	If $f\in A^2(\{\Psi_1<0\})$ and $c_o^{F}(\psi)<+\infty$, then for any $r\in (0,1]$,
	\[\frac{1}{r^2}\mu(\{c_o^{F}(\psi)\psi-\log|F|<\log r\})\geq G(0;\Psi_1,I_+(\Psi_1)_o,1)>0,\]
	where $\Psi_1:=\min\{2c_o^F(\psi)\psi-2\log|F|,0\}$, and $c_o^{F}(\psi):=\sup\{c\geq 0 : |F|^2e^{-2c\psi}$ is locally $L^1$ near $o \}$.
\end{Corollary}

Let $\varphi_0$ be a plurisubharmonic function on $D$, and let $f$ be a holomorphic function on $\{\Psi<0\}$. Denote that $a_o^f(\Psi;\varphi_0):=\sup \{a\geq 0 : f_o\in I(2a\Psi+\varphi_0)_o\}$, 
\[I_+(a\Psi_1+\varphi_0)_o:=\bigcup_{a'>a}I(a'\Psi_1+\varphi_0)_o\]
for any $a\geq 0$, and
\[C(\Psi,\varphi_0,J,f):=\inf\left\{\int_{\{\Psi<0\}}|\tilde{f}|^2e^{-\varphi_0} :  (\tilde{f}-f)_o\in J \ \& \ \tilde{f}\in\mathcal{O}(\{\Psi<0\})\right\}\]
for any $\mathcal{O}_o-$submodule $J$ of $I(\varphi_0)_o$. The following effectiveness result of strong openness property of the module $I(a\Psi+\varphi_0)_o$  can be reproved by Theorem \ref{increasing}.

\begin{Corollary}[\cite{GMY-BC2}]\label{SOPE}
	Let $C_1$ and $C_2$ be two positive constants. If
	
	(1) $\int_{\{\Psi<0\}}|f|^2e^{-\varphi_0-\Psi}\leq C_1$;
	
	(2) $C(\Psi,\varphi_0,I_+(2a_o^f(\Psi;\varphi_0)\Psi+\varphi_0)_o,f)\geq C_2$,\\	
	then for any $q>1$ satisfying
	\[\theta(q)>\frac{C_1}{C_2},\]
	we have $f_o\in I(q\Psi+\varphi_0)_o$, where $\theta(q)=\frac{q}{q-1}$.
\end{Corollary}

\section{Preparations}
\subsection{$L^2$ methods}
\

We recall the optimal $L^2$ extension theorem.

Let $\Omega$ be a pseudoconvex domain in $\mathbb{C}^{n+1}$ with coordinate $(z,t)$, where $z\in\mathbb{C}^n$, $t\in\mathbb{C}$. Let $p$ be the natural projections $p(z,t)=t$ on $\Omega$. Denote that $\omega:=p(\Omega)$ and $\Omega_t:=p^{-1}(t)$ for any $t\in\omega$. Let $\varphi$ be a plurisubharmonic function on $\Omega$.

\begin{Lemma}[Optimal $L^2$ extension theorem (\cite{Blocki}, see \cite{guan-zhou12,guan-zhou13ap,guan-zhou13p})]\label{L2ext}
	Let $\omega=\Delta_{t_0,r}$ be the disc in the complex plane centered at $t_0$ with radius $r$. Then for any $f$ in $A^2(\Omega_{t_0},e^{-\varphi})$, there exists a holomorphic function $\tilde{f}$ on $\Omega$, such that $\tilde{f}|_{\Omega_{t_0}}=f$, and
	\[\frac{1}{\pi r^2}\int_{\Omega}|\tilde{f}|^2e^{-\varphi}\leq\int_{\Omega_{t_0}}|f|^2e^{-\varphi}.\]
\end{Lemma}

The following $L^2$ method will be used to prove Theorem \ref{increasing}.

Let $D$ be a pseudoconvex domain in $\mathbb{C}^n$, and the origin $o\in D$. Let $F\not\equiv 0$ be a holomorphic function on $D$, and $\psi$ be a negative plurisubharmonic function on $D$. Let $\varphi_0$ be a plurisubharmonic function on $D$. Denote that
\[\varphi:=\varphi_0+2\max\{\psi,2\log|F|\},\]
and
\[\Psi:=\min\{\psi-2\log|F|, 0\}.\]
If $F(w)=0$ for $w\in D$, set $\Psi(w)=0$.

\begin{Lemma}[see \cite{guan-zhou13ap,GZeff,BGY,GMY-BC2}]\label{L2mthod}
	Let $t_0\in (0,+\infty)$ be arbitrary given. Let $f$ be a holomorphic function on $\{\Psi<-t_0\}$ such that
	\[\int_{\{\Psi<-t_0\}\cap K}|f|^2e^{-\varphi_0}<+\infty\]
	for any compact subset $K\subset D$. Then there exists a holomorphic function $\tilde{F}$ on $D$ such that
\[\int_D|\tilde{F}-(1-b_{t_0}(\Psi))fF^2|^2e^{-\varphi+v_{t_0}(\Psi)-\Psi}\leq C\int_D\mathbb{I}_{\{-t_0-1<\Psi<-t_0\}}|f|^2e^{-\varphi_0-\Psi},\]
	where $b_{t_0}(t)=\int_{-\infty}^t\mathbb{I}_{\{-t_0-1<s<-t_0\}}\mathrm{d}s$, $v_{t_0}(t)=\int_0^tb_{t_0}(s)\mathrm{d}s$ and $C$ is a positive constant.
\end{Lemma}
	
\subsection{Some lemmas about submodules of $I(\varphi_0)_o$}
\

Recall that $D$ is a pseudoconvex domain in $\mathbb{C}^n$, and the origin $o\in D$. Let $F\not\equiv 0$ be a holomorphic function on $D$, and $\psi$ be a negative plurisubharmonic function on $D$. Let $\varphi_0$ be a plurisubharmonic function on $D$. Denote that
\[\Psi:=\min\{\psi-2\log|F|, 0\}.\]
If $F(w)=0$ for $w\in D$, set $\Psi(w)=0$. We recall the following lemma.

\begin{Lemma}[\cite{GMY-BC2}]\label{l:converge}
	Let $J_o$ be an $\mathcal{O}_{\mathbb{C}^n,o}-$submodule of $I(\varphi_0)_o$ such that $I(\Psi+\varphi_0)_o\subset J_o$. Assume that $f_o\in J(\Psi)_o$. Let $U_0$ be a Stein open neighborhood of $o$. Let $\{f_j\}_{j\geq 1}$ be a sequence of holomorphic functions on $U_0\cap\{\Psi<-t_j\}$ for any $j\geq 1$, where $t_j\in (T,+\infty)$. Assume that $t_0=\lim_{j\rightarrow+\infty}t_j\in [T,+\infty)$,
	\[\limsup_{j\rightarrow+\infty}\int_{U_0\cap\{\Psi<-t_j\}}|f_j|^2e^{-\varphi_0}\leq C<+\infty,\]
	and $(f_j-f)_o\in J_o$. Then there exists a subsequence of $\{f_j\}_{j\geq 1}$ compactly convergent to a holomorphic function $f_0$ on $\{\Psi<-t_0\}\cap U_0$ which satisfies
	\[\int_{U_0\cap\{\Psi<-t_0\}}|f_0|^2e^{-\varphi_0}\leq C,\]
	and $(f_0-f)_o\in J_o$.
\end{Lemma}

It is well-known that $A^2(\{\Psi<0\},e^{-\varphi_0})$ is a Hilbert space. Let $J$ be an $\mathcal{O}_{o}-$submodule of $I(\varphi_0)_o$. We state that $A^2(\{\Psi<0\},e^{-\varphi_0})\cap J:=\{f\in A^2(\{\Psi<0\},e^{-\varphi_0}) : f_o\in J\}$ is a closed subspace of $A^2(\{\Psi<0\},e^{-\varphi_0})$ if $J\supset I(\Psi+\varphi_0)_o$.

\begin{Lemma}\label{Jclosed}
	If $J\supset I(\Psi+\varphi_0)_o$, then $A^2(\{\Psi<0\},e^{-\varphi_0})\cap J$ is closed in $A^2(\{\Psi<0\},e^{-\varphi_0})$.
\end{Lemma}

\begin{proof}
	Let $\{f_j\}$ be a sequence of holomorphic functions in $A^2(\{\Psi<0\},e^{-\varphi_0})\cap J$, such that $\lim_{j\rightarrow+\infty}f_j=f_0$ under the topology of $A^2(\{\Psi<0\},e^{-\varphi_0})$. Then $\{f_j\}$ compactly converges to $f_0$ on $\{\Psi<0\}$, and $(f_j-0)_o\in J$ for any $j$. According to Lemma \ref{l:converge}, we can get that $(f_0-0)_o\in J$, which means that $(f_0)_o\in J$, i.e. $f_0\in A^2(\{\Psi<0\},e^{-\varphi_0})\cap J$. The we know that $A^2(\{\Psi<0\},e^{-\varphi_0})\cap J$ is closed in $A^2(\{\Psi<0\},e^{-\varphi_0})$.
\end{proof}

The following two lemmas can be referred to \cite{BGY} or \cite{GMY-BC2}.

\begin{Lemma}[\cite{GMY-BC2}]\label{I=I+}
For any $a\geq 0$, there exists $a'>a$ such that $I(a'\Psi+\varphi_0)_o=I_+(a\Psi+\varphi_0)_o$.
\end{Lemma}

Let $f$ be a holomorphic function on $D$. Denote that
\[\Psi_1:=\min\{2c_o^{fF}(\psi)\psi-2\log|F|,0\},\]
where $c_o^{fF}(\psi):=\sup\{c\geq 0 : |fF|^2e^{-2c\psi}$ is locally $L^1$ near $o \}$.

\begin{Lemma}[\cite{BGY}, see also \cite{GMY-BC2}]\label{fonotinI+}
	$f_o\notin I_+(\Psi_1)_o$.
\end{Lemma}

\subsection{Some lemmas about functionals on $A^2(\{\Psi<0\},e^{-\varphi_0})$}
\

The following two lemmas will be used in the proof of Theorem \ref{concavity}. For the convenience of readers, we recall the proofs.

\begin{Lemma}\label{fjtof0}
	Let $D$ be a domain in $\mathbb{C}^n$, and let $\varphi_0$ be a plurisubharmonic function on $D$. Let $\{f_j\}$ be a sequence in $A^2(D,e^{-\varphi_0})$, such that $\int_{D}|f_j|^2e^{-\varphi_0}$ is uniformly bounded for any $j\in\mathbb{N}_+$. Assume that $f_j$ compactly converges to $f_0\in A^2(D,e^{-\varphi_0})$. Then for any $\xi\in A^2(D,e^{-\varphi_0})^*$,
	\[\lim_{j\rightarrow+\infty}\xi\cdot f_j=\xi\cdot f_0.\]
\end{Lemma}

\begin{proof}
	For any $f\in A^2(D,e^{-\varphi_0})$, denote that $\|f\|^2:=\int_{D}|f|^2e^{-\varphi_0}$. Let $\{f_{k_j}\}$ be any subsequence of $\{f_j\}$. Since $A^2(D,e^{-\varphi_0})$ is a Hilbert space, and $\|f_{k_j}\|^2$ is uniformly bounded, there exists a subsequence of $\{f_{k_j}\}$ (denoted by $\{f_{k_{l_j}}\}$) weakly convergent to some $\tilde{f}\in A^2(D,e^{-\varphi_0})$. Note that for any $z\in D$, the functional $e_z\in A^2(D,e^{-\varphi_0})^*$, where
	\begin{flalign*}
		\begin{split}
			e_z \ : \ A^2(D,e^{-\varphi_0})&\longrightarrow\mathbb{C}\\
			f&\longmapsto f(z).
		\end{split}
	\end{flalign*}
	Then we have
	\[f_0(z)=\lim_{j\rightarrow+\infty}e_z\cdot f_j=\lim_{j\rightarrow+\infty}e_z\cdot f_{k_{l_j}}=e_z\cdot\tilde{f}=\tilde{f}(z), \ \forall z\in D,\]
	thus $f_0=\tilde{f}$. It means that $\{f_{k_j}\}$ has a subsequence weakly convergent to $f_0$. Since $\{f_{k_j}\}$ is an arbitrary subsequence of $\{f_j\}$, we get that $\{f_j\}$ weakly converges to $f_0$. In other words, for any $\xi\in A^2(D,e^{-\varphi_0})^*$,
	\[\lim_{j\rightarrow+\infty}\xi\cdot f_j=\xi\cdot f_0.\]
\end{proof}

Let $\Omega:=D\times\omega\subset\mathbb{C}^{n+1}$, where $D$ is a domain in $\mathbb{C}^n$, $\omega$ is a domain in $\mathbb{C}$. Denote the coordinate on $\Omega$ by $(z,\tau)$, where $z\in D$, $\tau\in \omega$.  Let $\varphi_0$ be a plurisubharmonic function on $D$. Let $f$ be a holomorphic function on $\Omega$, such that
\[\int_{\Omega}|f(z,\tau)|^2e^{-\varphi_0(z)}<+\infty.\]
Denote $f_{\tau}:=f|_{D\times\{\tau\}}$.

\begin{Lemma}\label{xihol}
	For any $\xi\in A^2(D,e^{-\varphi_0})^*$, $\xi\cdot f_{\tau}$ is holomorphic with respect to $\tau\in\omega$.
\end{Lemma}

\begin{proof}
	We only need to prove that $h(\tau):=\xi\cdot f_{\tau}$ is holomorphic near any $\tau_0\in\omega$. Since $\tau_0\in \omega$, there exists $r>0$ such that $\Delta(\tau_0,2r)\subset\subset\omega$. Then for any $\tau\in\Delta(\tau_0,r)$, according to sub-mean value inequality of subharmonic functions, we have
	\[\int_D|f_{\tau}(z)|^2e^{-\varphi_0(z)}\leq \frac{1}{\pi r^2}\int_{D\times\Delta(t,r)}|f(z,\tau)|^2e^{-\varphi_0(z)}\leq\frac{1}{\pi r^2}\int_{\Omega}|f|^2e^{-\varphi_0}<+\infty,\]
	which implies that $f_{\tau}\in A^2(D,e^{-\varphi_0})$ and there exists $M>0$ such that $\int_D|f_{\tau}|^2e^{-\varphi_0}\leq M$ for any $\tau\in\Delta(\tau_0,r)$.
	
	Fix $z_0\in D$. According to Lemma \ref{Tvarphi} in Appendix, we can find a sequence $\{\xi_k\}\subset \ell_0\subset A^2(D,e^{-\varphi_0})^*$, such that
	\[\lim_{k\rightarrow+\infty}\|\xi_k-\xi\|_{A^2(D,e^{-\varphi_0})^*}=0,\]
	where
	\[\ell_0:=\{\eta=(\eta_{\alpha})_{\alpha\in\mathbb{N}^n} : \exists k\in\mathbb{N}, \text{such\ that\ } \eta_{\alpha}=0, \ \forall |\alpha|\geq k\}.\]
	Here for any $\eta=(\eta_{\alpha})_{\alpha\in\mathbb{N}^n}\in\ell_0$ and $f\in A^2(D,e^{-\varphi_0})$, define that
	\[\eta\cdot f:=\sum_{\alpha\in\mathbb{N}^n}\eta_{\alpha}\frac{f^{(\alpha)}(z_0)}{\alpha!}.\]
	Note that for any $\xi_k\in \ell_0$, $\xi_k\cdot f_{\tau}$ can be written as
	\[\xi_k\cdot f_{\tau}=\sum_{\alpha\in\mathbb{N}^n, \ |\alpha|\leq l_k}c_{\alpha,k}\frac{\partial^{\alpha}f(z,\tau)}{\partial z^{\alpha}}(z_0,\tau),\]
	where $l_k$ is a finite integer, and $c_{\alpha,k}\in\mathbb{C}$ are constants. It is clear that $h_k(\tau):=\xi_k\cdot f_{\tau}$ is holomorphic with respect to $\tau\in\omega$ for any $k\in\mathbb{N}_+$, since any $\frac{\partial^{\alpha}f(z,\tau)}{\partial z^{\alpha}}(z_0,\tau)$ is holomorphic with respect to $\tau$ for any $\alpha\in\mathbb{N}_+$. Note that for any $\tau\in\Delta(\tau_0,r)$, we have
	\begin{flalign*}
		\begin{split}
			&|h_k(\tau)-h(\tau)|^2\\
			=&|(\xi_k-\xi)\cdot f_{\tau}|^2\\
			\leq&\|\xi_k-\xi\|^2_{A^2(D,e^{-\varphi_0})^*}\int_D|f_{\tau}|^2e^{-\varphi_0}\\
			\leq&M\|\xi_k-\xi\|^2_{A^2(D,e^{-\varphi_0})^*},
		\end{split}
	\end{flalign*}
	which means that $h_k$ uniformly converges to $h$ on $\Delta(\tau_0,r)$. According to Weierstrass theorem, we know that $h$ is holomorphic on $\Delta(\tau_0,r)$, i.e. near $\tau_0$. Then we get that $\xi\cdot f_{\tau}$ is holomorphic with respect to $\tau\in\omega$.
\end{proof}

\subsection{Some properties of $K^{\varphi_0}_{\xi,\Psi,\lambda}(t)$}
\

In this section, we prove some properties of the Bergman kernel $K^{\varphi_0}_{\xi,\Psi,\lambda}(t)$.

Let $\xi\in A^2(\{\Psi<0\},e^{-\varphi_0})^*\setminus\{0\}$.

\begin{Lemma}\label{sup=max}
	For any $t\in [0,+\infty)$, if $K^{\varphi_0}_{\xi,\Psi,\lambda}(t)\in (0,+\infty)$, then there exists $\tilde{f}\in A^2(\{\Psi<0\},e^{-\varphi_0})$, such that
	\[K^{\varphi_0}_{\xi,\Psi,\lambda}(t)=\frac{|\xi\cdot \tilde{f}|^2}{\|\tilde{f}\|_{\lambda,t}^2}.\]
\end{Lemma}

\begin{proof}
	By the definition of $K^{\varphi_0}_{\xi,\Psi,\lambda}(t)$, there exists a sequence $\{f_j\}$ of holomorphic functions in $A^2(\{\Psi<0\},e^{-\varphi_0})$, such that $\|f_j\|_{\lambda,t}=1$, and $\lim_{j\rightarrow+\infty}|\xi\cdot f_j|^2=K^{\varphi_0}_{\xi,\Psi,\lambda}(t)$. Then $\int_{\{\Psi<0\}}|f_j|^2e^{-\varphi_0}$ is uniformly bounded. Following from Montel's theorem, we can get a subsequence of $\{f_j\}$ compactly convergent to a holomorphic function $\tilde{f}$ on $\{\Psi<0\}$. According to Fatou's lemma, we have $\|\tilde{f}\|_{\lambda,t}\leq 1$, and according to Lemma \ref{fjtof0}, we have $|\xi\cdot \tilde{f}|^2=K^{\varphi_0}_{\xi,\Psi,\lambda}(t)$, thus $K^{\varphi_0}_{\xi,\Psi,\lambda}(t)\leq\frac{|\xi\cdot \tilde{f}|^2}{\|\tilde{f}\|_{\lambda,t}^2}$. Note that $\|\tilde{f}\|_{\lambda,t}\leq 1$ implies $\tilde{f}\in A^2(\{\Psi<0\},e^{-\varphi_0})$, which means $K^{\varphi_0}_{\xi,\Psi,\lambda}(t)\geq\frac{|\xi\cdot \tilde{f}|^2}{\|\tilde{f}\|_{\lambda,t}^2}$. We get that $K^{\varphi_0}_{\xi,\Psi,\lambda}(t)=\frac{|\xi\cdot \tilde{f}|^2}{\|\tilde{f}\|_{\lambda,t}^2}$.
\end{proof}

Let $J$ be an $\mathcal{O}_{o}-$submodule of $I(\varphi_0)_o$ such that $J\supset I(\Psi+\varphi_0)_o$, and let $f\in A^2(\{\Psi<0\},e^{-\varphi_0})$, such that $f_o\notin J$. Recall the minimal $L^2$ integral (\cite{BGY,GMY-BC2})
\[C(\Psi,\varphi_0,J,f):=\inf\left\{\int_{\{\Psi<0\}}|\tilde{f}|^2e^{-\varphi_0} :  (\tilde{f}-f)_o\in J \ \& \ \tilde{f}\in\mathcal{O}(\{\Psi<0\})\right\}.\]
Then the following lemma holds.

\begin{Lemma}\label{B=C}
	Assume that $C(\Psi,\varphi_0,J,f)\in (0,+\infty)$, then
	\begin{equation}
		C(\Psi,\varphi_0,J,f)=\sup_{\substack{\xi\in A^2(\{\Psi<0\},e^{-\varphi_0})^*\setminus\{0\}\\ \xi|_{A^2(\{\Psi<0\},e^{-\varphi_0})\cap J}\equiv 0}}\frac{|\xi\cdot f|^2}{K^{\varphi_0}_{\xi,\Psi,\lambda}(0)}.
	\end{equation}
\end{Lemma}

\begin{proof}
	Note that $\xi\cdot\tilde{f}=\xi\cdot f$ for any $\tilde{f}\in A^2(\{\Psi<0\},e^{-\varphi_0})$ with $(\tilde{f}-f)_o\in J$ and $\xi\in A^2(\{\Psi<0\},e^{-\varphi_0})^*$ satisfying $\xi|_{A^2(\{\Psi<0\},e^{-\varphi_0})\cap J}\equiv 0$. Then we have
	\begin{flalign*}
		\begin{split}
		K^{\varphi_0}_{\xi,\Psi,\lambda}(0)&=\sup_{h\in A^2(\{\Psi<0\},e^{-\varphi_0})}\frac{|\xi\cdot h|^2}{\int_{\{\Psi<0\}}|h|^2e^{-\varphi_0}}\\
		&\geq\sup_{\substack{\tilde{f}\in A^2(\{\Psi<0\},e^{-\varphi_0})\\ (\tilde{f}-f)_o\in J}}\frac{|\xi\cdot \tilde{f}|^2}{\int_{\{\Psi<0\}}|\tilde{f}|^2e^{-\varphi_0}}\\
		&= \sup_{\substack{\tilde{f}\in A^2(\{\Psi<0\},e^{-\varphi_0})\\ (\tilde{f}-f)_o\in J}}\frac{|\xi\cdot f|^2}{\int_{\{\Psi<0\}}|\tilde{f}|^2e^{-\varphi_0}}.
	\end{split}
	\end{flalign*}
	Thus we get that
	\begin{flalign*}
		\begin{split}
			&\sup_{\substack{\xi\in A^2(\{\Psi<0\},e^{-\varphi_0})^*\setminus\{0\}\\ \xi|_{A^2(\{\Psi<0\},e^{-\varphi_0})\cap J}\equiv 0}}\frac{|\xi\cdot f|^2}{K^{\varphi_0}_{\xi,\Psi,\lambda}(0)}\\
			\leq&\inf_{\substack{\tilde{f}\in A^2(\{\Psi<0\},e^{-\varphi_0})\\ (\tilde{f}-f)_o\in J}}\int_{\{\Psi<0\}}|\tilde{f}|^2e^{-\varphi_0}\\
			=&C(\Psi,\varphi_0,J,f).
			\end{split}
	\end{flalign*}
	
	Since $A^2(\{\Psi<0\},e^{-\varphi_0})$ is a Hilbert space, and $A^2(\{\Psi<0\},e^{-\varphi_0})\cap J$ is a closed proper subspace of $A^2(\{\Psi<0\},e^{-\varphi_0})$ (using Lemma \ref{Jclosed}), there exists a closed subspace $H$ of $A^2(\{\Psi<0\},e^{-\varphi_0})$ such that $H=(A^2(\{\Psi<0\},e^{-\varphi_0})\cap J)^{\bot}\neq\{0\}$. Then for $f\in A^2(\{\Psi<0\},e^{-\varphi_0})$, we can make the decomposition $f=f_J+f_H$, such that $f_J\in A^2(\{\Psi<0\},e^{-\varphi_0})\cap J$, and $f_H\in H$. Note that the linear functional $\xi_f$ defined as follows:
	\[\xi_f\cdot g:=\int_{\{\Psi<0\}}g\overline{f_H}e^{-\varphi_0}, \ \forall g\in A^2(\{\Psi<0\},e^{-\varphi_0}),\]
	satisfies that $\xi_f\in A^2(\{\Psi<0\},e^{-\varphi_0})^*\setminus\{0\}$ and $\xi_f|_{A^2(\{\Psi<0\},e^{-\varphi_0})\cap J}\equiv 0$. Then we have
	\[\sup_{\substack{\xi\in A^2(\{\Psi<0\},e^{-\varphi_0})^*\setminus\{0\}\\ \xi|_{A^2(\{\Psi<0\},e^{-\varphi_0})\cap J}\equiv 0}}\frac{|\xi\cdot f|^2}{K^{\varphi_0}_{\xi,\Psi,\lambda}(0)}\\
	\geq\frac{|\xi_f\cdot f|^2}{K^{\varphi_0}_{\xi_f,\Psi,\lambda}(0)}.\]
	Besides, we can know that
	\[K^{\varphi_0}_{\xi_f,\Psi,\lambda}(0)=\sup_{h\in A^2(\{\Psi<0\},e^{-\varphi_0})}\frac{|\int_{\{\Psi<0\}} h\overline{f_H}e^{-\varphi_0}|^2}{\int_{\{\Psi<0\}}|h|^2e^{-\varphi_0}}\leq\int_{\{\Psi<0\}}|f_H|^2e^{-\varphi_0},\]
	and
	\[\xi_f\cdot f=\xi_f\cdot (f_J+f_H)=\xi_f\cdot f_H=\int_{\{\Psi<0\}}|f_H|^2e^{-\varphi_0}.\]
	Then we have
	\[\frac{|\xi_f\cdot f|^2}{K^{\varphi_0}_{\xi_f,\Psi,\lambda}(0)}\geq\int_{\{\Psi<0\}}|f_H|^2e^{-\varphi_0}\geq C(\Psi,\varphi_0,J,f),\]
	which implies that
	\[\sup_{\substack{\xi\in A^2(\{\Psi<0\},e^{-\varphi_0})^*\setminus\{0\}\\ \xi|_{A^2(\{\Psi<0\},e^{-\varphi_0})\cap J}\equiv 0}}\frac{|\xi\cdot f|^2}{K^{\varphi_0}_{\xi,\Psi,\lambda}(0)}\\
	\geq C(\Psi,\varphi_0,J,f).\]

	Lemma \ref{B=C} is proved.
\end{proof}

\section{Proof of Theorem \ref{concavity}}

We prove Theorem \ref{concavity} by using Lemma \ref{L2ext} (optimal $L^2$ extension theorem).

\begin{proof}[Proof of Theorem \ref{concavity}]
	Denote that $\Omega:=\{\Psi<0\}\times E=\{\Psi<0\}\times\{w\in \mathbb{C}: \text{Re\ }w\geq 0\}$, and the coordinate on $\Omega$ is $(z,w)$, where $z\in \{\Psi<0\}\subset \mathbb{C}^n$, $w\in E=\{w\in \mathbb{C}: \text{Re\ }w\geq 0\}$. Note that $D\setminus\{F=0\}$ is a pseudoconvex domain in $\mathbb{C}^n$, and $\{\Psi<0\}=\{\psi+2\log|1/F|<0\}$ on $D\setminus\{F=0\}$. Then $\{\Psi<0\}$ is a pseudoconvex domain in $\mathbb{C}^n$, and $\Psi$ is a plurisubharmonic function on $\{\Psi<0\}$. We get that $\Omega$ is a pseudoconvex domain in $\mathbb{C}^{n+1}$. For any $(z,w)\in \Omega$, let
	\[\tilde{\Psi}(z,w):=\varphi_0(z)+\Psi_{\lambda,\text{Re\ }w}=\varphi_0(z)+\lambda\max\{\Psi(z)+\text{Re\ }w,0\}.\]
	Then $\tilde{\Psi}$ is a plurisubharmonic function on $\Omega$.
	
	Denote that
	\[K(w):=K^{\varphi_0}_{\xi,\Psi,\lambda}(\text{Re\ }w)\]
	for any $w\in E$. We prove that $\log K(w)$ is a subharmonic function with respect to $w\in E$.
	
	Firstly we prove that $\log K(w)$ is upper semicontinuous. Let $w_j\in E$ such that $\lim_{j\rightarrow+\infty}w_j=w_0\in E$. We assume that $\{w_{k_j}\}$ is the subsequence of $\{w_j\}$ such that
	\[\lim_{j\rightarrow+\infty}K(w_{k_j})=\limsup_{j\rightarrow+\infty}K(w_j).\]
	By Lemma \ref{sup=max}, there exists a sequence of holomorphic functions $\{f_j\}$ on $\{\Psi<0\}$ such that $f_j\in A^2(\{\Psi<0\},e^{-\varphi_0})$, $\|f_j\|_{\lambda,\text{Re\ }w_j}=1$, and $|\xi\cdot f_j|^2=K(w_j)$, for any $j\in\mathbb{N}_+$. Since $\{w_j\}$ is bounded in $\mathbb{C}$, there exists some $s_0>0$, such that $\text{Re\ }w_j<s_0$ for any $j$, which implies that
	\[\int_{\{\Psi<0\}}|f_j|^2e^{-\varphi_0}\leq e^{\lambda s_0}\|f_j\|^2_{\lambda,\text{Re\ }w_j}=e^{\lambda s_0}, \ \forall j\in\mathbb{N}_+.\]
	Then following from Montel's theorem, we can get a subsequence of $\{f_{k_j}\}$ (denoted by $\{f_{k_j}\}$ itself) compactly convergent to a holomorphic function $f_0$ on $\{\Psi<0\}$. According to Fatou's lemma, we have
	\begin{flalign*}
		\begin{split}
			\|f_0\|_{\lambda,\text{Re\ }w_0}&=\int_{\{\Psi<0\}}|f_0(z)|^2e^{-\varphi_0(z)-\lambda\max\{\Psi(z)+\text{Re\ }w_0,0\}}\\
			&=\int_{\{\Psi<0\}}\lim_{j\rightarrow+\infty}|f_{k_j}(z)|^2e^{-\varphi_0(z)-\lambda\max\{\Psi(z)+\text{Re\ }w_{k_j},0\}}\\
			&\leq\liminf_{j\rightarrow+\infty}\int_{\{\Psi<0\}}|f_{k_j}(z)|^2e^{-\varphi_0(z)-\lambda\max\{\Psi(z)+\text{Re\ }w_{k_j},0\}}\\
			&=\liminf_{j\rightarrow+\infty}\|f_{k_j}\|_{\lambda,\text{Re\ }w_j}=1.
		\end{split}
	\end{flalign*}
	Then $\int_{\{\Psi<0\}}|f_0|^2e^{-\varphi_0}\leq e^{\lambda\text{Re\ }w_0} \|f_0\|^2_{\lambda,\text{Re\ }w_0}\leq e^{\lambda s_0}<+\infty$, which implies that $f_0\in A^2(\{\Psi<0\},e^{-\varphi_0})$. Lemma \ref{fjtof0} shows that $|\xi\cdot f_0|^2=\lim_{j\rightarrow+\infty}|\xi\cdot f_{k_j}|^2=\limsup_{j\rightarrow+\infty}K(w_j)$. Thus
	\[K(w_0)\geq\frac{|\xi\cdot f_0|^2}{\|f_0\|^2_{\lambda,\text{Re\ }w_0}}\geq\limsup_{j\rightarrow+\infty}K(w_j),\]
	which means
	\[\log K(w_0)\geq \limsup_{j\rightarrow+\infty}\log K(w_j).\]
	Then we get that $\log K(w)$ is upper semicontinuous with respect to $w\in E$.
	
	Secondly we prove that $\log K(w)$ satisfies the sub-mean value inequality.
	
	Let $\Delta(w_0,r)\subset E$ be the disc centered at $w_0$ with radius $r$, and let $\Omega':=\{\Psi<0\}\times\Delta(w_0,r)\subset\mathbb{C}^{n+1}$. Let $f_0\in A^2(\{\Psi<0\},e^{-\varphi_0})$ such that
	\[K(w_0)=\frac{|\xi\cdot f_0|^2}{\|f_0\|^2_{\lambda,\text{Re\ }w_0}}\]
	by Lemma \ref{sup=max}.
	
	Note that  $\Omega'$ is a pseudoconvex domain in $\mathbb{C}^{n+1}$, and $\tilde{\Psi}(z,w)=\varphi_0(z)+\Psi_{\lambda,\text{Re\ }w}=\varphi_0(z)+\lambda\max\{\Psi(z)+\text{Re\ }w,0\}$ is a plurisubharmonic function on $\Omega'$. Using Lemma \ref{L2ext} (optimal $L^2$ extension theorem), we can get a holomorphic function $\tilde{f}$ on $\Omega'$ such that $\tilde{f}(z,w_0)=f_0(z)$ for any $z\in\{\Psi<0\}$, and
	\begin{equation}\label{ineqL2ext}
		\frac{1}{\pi r^2}\int_{\Omega'}|\tilde{f}(z,w)|^2e^{-\tilde{\Psi}(z,w)}\leq \int_{\{\Psi<0\}}|f_0(z)|^2e^{-\Psi(z,w_0)}.
	\end{equation}
	Denote that $\tilde{f}_w(z)=\tilde{f}(z,w)=\tilde{f}|_{\{\Psi<0\}\times\{w\}}$. Since the function $y=\log x$ is concave, according to Jensen's inequality and inequality (\ref{ineqL2ext}), we have
	\begin{flalign}\label{ineqJensen}
		\begin{split}
			\log\|f_0\|^2_{\lambda,\text{Re\ }w_0}&=\log\left(\int_{\{\Psi<0\}}|f_0(z)|^2e^{-\Psi(z,w_0)}\right)\\
			&\geq\log\left(\frac{1}{\pi r^2}\int_{\Omega'}|\tilde{f}(z,w)|^2e^{-\tilde{\Psi}(z,w)}\right)\\
			&=\log\left(\frac{1}{\pi r^2}\int_{\Delta(w_0,r)}\int_{\{\Psi<0\}\times\{w\}}|\tilde{f}_w(z)|^2e^{-\tilde{\Psi}(z,w)}\right)\\
			&\geq\frac{1}{\pi r^2}\int_{\Delta(w_0,r)}\log\left(\|\tilde{f}_w\|^2_{\lambda,\text{Re\ }w}\right)\\
			&\geq\frac{1}{\pi r^2}\int_{\Delta(w_0,r)}\left(\log|\xi\cdot \tilde{f}_w|^2-\log K(w)\right).
		\end{split}
	\end{flalign}
	It follows from Lemma \ref{xihol} that $\xi\cdot \tilde{f}_w$ is holomorphic with respect to $w$, which implies that $\log|\xi\cdot\tilde{f}_w|^2$ is subharmonic with respect to $w$. Combining with $\tilde{f}_{w_0}=f_0$, we have
	\[\log|\xi\cdot f_0|^2\leq\frac{1}{\pi r^2}\int_{\Delta(w_0,r)}\log|\xi\cdot \tilde{f}_w|^2.\]
	Combining with inequality (\ref{ineqJensen}), we get
	\[\log\|f_0\|^2_{\lambda,\text{Re\ }w_0}\geq\log|\xi\cdot f_0|^2-\frac{1}{\pi r^2}\int_{\Delta(w_0,r)}\log K(w),\]
	which means
	\[\log K(w_0)\leq\frac{1}{\pi r^2}\int_{\Delta(w_0,r)}\log K(w).\]
	
	Since $\log K(w)$ is upper semicontinuous and satisfies the sub-mean value inequality, we know that $\log K(w)$ is a subharmonic function on the interior of $E$. In addition, since $\log K(w)$ is upper semicontinuous near $\{0\}+\sqrt{-1}\mathbb{R}$, and $\log K(w)$ is only dependent on the real part of $w$, we know that $\log K(w)$ is a subharmonic function on $E$.
\end{proof}

\section{Proofs of Theorem \ref{increasing} and Remark \ref{notconstant}}
In this section, we give the proofs of Theorem \ref{increasing} and Remark \ref{notconstant}. We need the following lemma.

\begin{Lemma}[see \cite{Demaillybook}]\label{Re}
	Let $D=I+\sqrt{-1}\mathbb{R}:=\{z=x+\sqrt{-1}y\in\mathbb{C} : x\in I, y\in\mathbb{R}\}$ be a subset of $\mathbb{C}$, where $I$ is an interval in $\mathbb{R}$. Let $\phi(z)$ be a subharmonic function on $D$ which is only dependent on $x=\text{Re\ }z$. Then $\phi(x):=\phi(x+\sqrt{-1}\mathbb{R})$ is a convex function with respect to $x\in I$.
\end{Lemma}

\begin{proof}[Proof of Theorem \ref{increasing}]
	It follows from Theorem \ref{concavity} that $\log K^{\varphi_0}_{\xi,\Psi,\lambda}(\text{Re\ } w)$ is subharmonic with respect to $w\in [0,+\infty)+\sqrt{-1}\mathbb{R}$. Note that $\log K^{\varphi_0}_{\xi,\Psi,\lambda}(\text{Re\ } w)$ is only dependent on $\text{Re\ }w$, then following from Lemma \ref{Re}, we get that $\log K^{\varphi_0}_{\xi,\Psi,\lambda}(t)=\log K^{\varphi_0}_{\xi,\Psi,\lambda}(t+\sqrt{-1}\mathbb{R})$ is convex with respect to $t\in [0,+\infty)$, which implies that $-\log K^{\varphi_0}_{\xi,\Psi,\lambda}(t)+t$ is concave with respect to $t\in [0,+\infty)$. Then for any $\xi\in A^2(\{\Psi<0\},e^{-\varphi_0})^*$ with $\xi|_{A^2(\{\Psi<0\},e^{-\varphi_0})\cap J }\equiv 0$, to prove that $\log -K^{\varphi_0}_{\xi,\Psi,\lambda}(t)+t$ is increasing, we only need to prove that $\log -K^{\varphi_0}_{\xi,\Psi,\lambda}(t)+t$ has a lower bound on $[0,+\infty)$.
	
	Using Lemma \ref{sup=max}, we obtain that there exists $f_t\in A^2(\{\Psi<0\},e^{-\varphi_0})$ for any $t\in [0,+\infty)$, such that $\xi\cdot f_t=1$ and
	\begin{equation}\label{K=ft}
		K^{\varphi_0}_{\xi,\Psi,\lambda}(t)=\frac{1}{\|f_t\|_{\lambda,t}^2}.
	\end{equation}
	In addition, according to Lemma \ref{L2mthod}, there exists a holomorphic function $\tilde{F}$ on $D$ such that
	\begin{equation}\label{tildeF}
		\int_D|\tilde{F}-(1-b_{t}(\Psi))f_tF^2|^2e^{-\varphi+v_{t}(\Psi)-\Psi}\leq C\int_D\mathbb{I}_{\{-t-1<\Psi<-t\}}|f_t|^2e^{-\varphi_0-\Psi},
	\end{equation}
	where
	\[\varphi=\varphi_0+2\max\{\psi,2\log|F|\},\]
	and $C$ is a positive constant. Then it follows from inequality (\ref{tildeF}) that
	\begin{flalign}\label{tildeF2}
		\begin{split}
			&\int_{\{\Psi<0\}}|\tilde{F}-(1-b_{t}(\Psi))f_tF^2|^2e^{-\varphi+v_{t}(\Psi)-\Psi}\\
			\leq&\int_D|\tilde{F}-(1-b_{t}(\Psi))f_tF^2|^2e^{-\varphi+v_{t}(\Psi)-\Psi}\\
			\leq&C\int_D\mathbb{I}_{\{-t-1<\Psi<-t\}}|f_t|^2e^{-\varphi_0-\Psi}\\
			\leq&Ce^{t+1}\int_{\{\Psi<-t\}}|f_t|^2e^{-\varphi_0}.
		\end{split}
	\end{flalign}
	Denote that $\tilde{F}_t:=\tilde{F}/F^2$ on $\{\Psi<0\}$, then $\tilde{F}_t$ is a holomorphic function on $\{\Psi<0\}$. Note that $|F|^4e^{-\varphi}=e^{-\varphi_0}$ on $\{\Psi<0\}$. Then inequality (\ref{tildeF2}) implies that
	\begin{equation}\label{tildeFt}
	\int_{\{\Psi<0\}}|\tilde{F}_t-(1-b_{t}(\Psi))f_t|^2e^{-\varphi_0+v_{t}(\Psi)-\Psi}\leq Ce^{t+1}\int_{\{\Psi<-t\}}|f_t|^2e^{-\varphi_0}<+\infty.
\end{equation}
According to inequality (\ref{tildeFt}), we can get that $(\tilde{F}_t-f_t)_o\in I(\Psi+\varphi_0)\subset J$, which means that $\xi\cdot \tilde{F}_t=\xi\cdot f_t=1$. Besides, since $v_t(\Psi)\geq \Psi$, we have
\begin{flalign*}
	\begin{split}
		&\left(\int_{\{\Psi<0\}}|\tilde{F}_t-(1-b_{t}(\Psi))f_t|^2e^{-\varphi_0+v_{t}(\Psi)-\Psi}\right)^{1/2}\\
		\geq&\left(\int_{\{\Psi<0\}}|\tilde{F}_t-(1-b_{t}(\Psi))f_t|^2e^{-\varphi_0}\right)^{1/2}\\
		\geq&\left(\int_{\{\Psi<0\}}|\tilde{F}_t|^2e^{-\varphi_0}\right)^{1/2}-\left(\int_{\{\Psi<0\}}|(1-b_{t}(\Psi))f_t|^2e^{-\varphi_0}\right)^{1/2}\\
		\geq&\left(\int_{\{\Psi<0\}}|\tilde{F}_t|^2e^{-\varphi_0}\right)^{1/2}-\left(\int_{\{\Psi<-t\}}|f_t|^2e^{-\varphi_0}\right)^{1/2}.
	\end{split}
\end{flalign*}
Combining with inequality (\ref{tildeFt}), we have
\begin{flalign*}
	\begin{split}
		&\int_{\{\Psi<0\}}|\tilde{F}_t|^2e^{-\varphi_0}\\
		\leq&2\int_{\{\Psi<0\}}|\tilde{F}_t-(1-b_{t}(\Psi))f_t|^2e^{-\varphi_0+v_{t}(\Psi)-\Psi}+2\int_{\{\Psi<-t\}}|f_t|^2e^{-\varphi_0}\\
		\leq&2(Ce^{t+1}+1)\int_{\{\Psi<-t\}}|f_t|^2e^{-\varphi_0}.
	\end{split}
\end{flalign*}
Note that
	\begin{flalign*}
	\begin{split}
		\|f_t\|^2_{\lambda,t}=&\int_{\{\Psi<0\}}|f_t|^2e^{-\varphi_0-\Psi_{\lambda,t}}\\
		=&\int_{\{\Psi<-t\}}|f_t|^2e^{-\varphi_0}+\int_{\{0>\Psi\geq -t\}}|f_t|^2e^{-\varphi_0-\lambda(\Psi+t)}\\
		\geq&\int_{\{\Psi<-t\}}|f_t|^2e^{-\varphi_0}.
	\end{split}
\end{flalign*}
Then we have
\[\int_{\{\Psi<0\}}|\tilde{F}_t|^2e^{-\varphi_0}\leq 2(Ce^{t+1}+1)\|f_t\|_{\lambda,t}^2=C_1\frac{e^t}{K^{\varphi_0}_{\xi,\Psi,\lambda}(t)},\]
where $C_1:=2(eC+1)$ is a positive constant. In addition, $\xi\cdot \tilde{F}_t=1$ implies that
\[\int_{\{\Psi<0\}}|\tilde{F}_t|^2e^{-\varphi_0}=\|\tilde{F}_t\|^2_{\lambda,0}\geq (K^{\varphi_0}_{\xi,\Psi,\lambda}(0))^{-1}.\]
Then we get that
\[-\log K^{\varphi_0}_{\xi,\Psi,\lambda}(t)+t\geq C_2,\ \forall t\in [0,+\infty),\]
where $C_2:=\log (C_1^{-1}K^{\varphi_0}_{\xi,\Psi,\lambda}(0))$ is a finite constant. Since $-\log K^{\varphi_0}_{\xi,\Psi,\lambda}(t)+t$ is concave, we get that $-\log K^{\varphi_0}_{\xi,\Psi,\lambda}(t)+t$ is increasing with respect to $t\in [0,+\infty)$.
\end{proof}

In the following we give the proof of Remark \ref{notconstant}.

\begin{proof}[Proof of Remark \ref{notconstant}]
	Denote that $K(t):=K^{\varphi_0}_{\xi,\Psi,\lambda}(t)$ for any $t\in [0,+\infty)$. According to Theorem \ref{concavity} and Lemma \ref{Re}, we can know that $\log K(t) -kt$ is convex on $[0,+\infty)$. Combining with that $e^{-kt} K(t)$ is increasing and not a constant function on $[0,T]$, which implies that $\log K(t)-kt=\log(e^{-kt} K(t))$ is increasing and not a constant function on $[0,T]$, we have that $\log K(t)-kt$ is strictly increasing on $[T,+\infty)$. Then $e^{-kt} K^{\varphi_0}_{\xi,\Psi,\lambda}(t)=\exp(\log K(t)-kt)$ is strictly increasing on $[T,+\infty)$.
\end{proof}

\section{Proof of Corollary \ref{J-M+f}}
In this section, we give the proof of Corollary \ref{J-M+f}.

\begin{proof}[Proof of Corollary \ref{J-M+f}]
For any $p\in (1,2)$, $\lambda>0$, let $\xi\in A^2(\{\Psi_1<0\})^*\setminus\{0\}$, such that $\xi|_{A^2(\{\Psi_1<0\})\cap J_p}\equiv 0$, where $J_p:=I(p\Psi_1)_o$. Denote that
\[K_{\xi,p,\lambda}(t):=\sup_{\tilde{f}\in A^2(\{\Psi_1<0\})}\frac{|\xi\cdot \tilde{f}|^2}{\|\tilde{f}\|^2_{p,\lambda,t}},\]
where
\[\|\tilde{f}\|_{p,\lambda,t}:=\left(\int_{\{\Psi_1<0\}}|\tilde{f}|^2e^{-\lambda\max\{p\Psi_1+t,0\}}\right)^{1/2},\]
and $t\in [0,+\infty)$. Note that
\[p\Psi_1=\min\{(2pc_o^{fF}(\psi)\psi+(4-2p)\log|F|)-2\log|F^2|,0\}\]
and Lemma \ref{fonotinI+} shows $f_o\notin J_p$, which implies that $A^2(\{\Psi_1<0\})\cap J_p$ is a proper subspace of $A^2(\{\Psi_1<0\})$, and $K_{\xi,p,\lambda}(0)\in (0,+\infty)$. Theorem \ref{increasing} tells us that $-\log K_{\xi,p,\lambda}(t)+t$ is increasing with respect to $t\in [0,+\infty)$, which implies that
\begin{equation}\label{-logK(t)+t}
	-\log K_{\xi,p,\lambda}(t)+t\geq -\log K_{\xi,p,\lambda}(0), \ \forall t\in [0,+\infty).
\end{equation}

Since $f\in A^2(\{\Psi_1<0\})$, following from inequality (\ref{-logK(t)+t}), we get that
\[\|f\|^2_{p,\lambda,t}\geq\frac{|\xi\cdot f|^2}{K_{\xi,p,\lambda}(t)}\geq e^{-t}\frac{|\xi\cdot f|^2}{K_{\xi,p,\lambda}(0)}, \ \forall t\in [0,+\infty).\]
In addition, since $f_o\notin J_p$, according to Lemma \ref{B=C}, we have
\begin{flalign}\label{f>e^-tC}
	\begin{split}
	\|f\|^2_{p,\lambda,t}&\geq\sup_{\substack{\xi\in A^2(\{\Psi_1<0\})^*\setminus\{0\}\\ \xi|_{A^2(\{\Psi_1<0\})\cap J_p}\equiv 0}}e^{-t}\frac{|\xi\cdot f|^2}{K_{\xi,p,\lambda}(0)}\\
	&=e^{-t}C(p\Psi_1,0,J_p,f), \ \forall t\in [0,+\infty).
\end{split}
\end{flalign}
Note that for any $t\in [0,+\infty)$,
\begin{equation}\label{f2}
	\|f\|^2_{p,\lambda,t}=\int_{\{p\Psi_1<-t\}}|f|^2+\int_{\{0>p\Psi_1\geq -t\}}|f|^2e^{-\lambda(p\Psi_1+t)}.
\end{equation}
Since for any $\lambda>0$,
\[\int_{\{0>p\Psi_1\geq t\}}|f|^2e^{-\lambda(p\Psi_1+t)}\leq\int_{\{0>p\Psi_1\geq -t\}}|f|^2<+\infty,\]
and $\lim_{\lambda\rightarrow+\infty}e^{-\lambda(p\Psi_1+t)}=0$ on $\{0>p\Psi_1\geq -t\}$, according to Lebesgue's dominated convergence theorem, we have
\[\lim_{\lambda\rightarrow+\infty}\int_{\{0>p\Psi_1\geq -t\}}|f|^2e^{-\lambda(p\Psi_1+t)}=0.\]
Then equality (\ref{f2}) implies
\begin{equation}
	\lim_{\lambda\rightarrow+\infty}\|f\|^2_{p,\lambda,t}=\int_{\{p\Psi_1<-t\}}|f|^2, \ \forall t\in [0,+\infty).
\end{equation}
Letting $\lambda\rightarrow+\infty$ in inequality (\ref{f>e^-tC}), we get that for any $t\in [0,+\infty)$,
\begin{equation}\label{f>e^-tC2}
	\int_{\{p\Psi_1<-t\}}|f|^2\geq e^{-t}C(p\Psi_1,0,J_p,f).
\end{equation}

Note that $\{p\Psi_1<0\}=\{\Psi_1<0\}$ and $J_p\subset I_+(\Psi_1)_o$ for any $p\in (1,2)$. Then we have
\[C(p\Psi_1,0,J_p,f)\geq C(\Psi_1,0,I_+(\Psi_1)_o,f), \ \forall p\in (1,2).\]
Since $\int_{\{\Psi_1<0\}}|f|^2<+\infty$, it follows from Lebesgue's dominated convergence theorem and inequality (\ref{f>e^-tC2}) that
\begin{flalign}
	\begin{split}
		&\int_{\{\Psi_1<-t\}}|f|^2\\
		=&\lim_{p\rightarrow 1+0}\int_{\{p\Psi_1<-t\}}|f|^2\\
		\geq&\limsup_{p\rightarrow 1+0}e^{-t}C(p\Psi_1,0,J_p,f)\\
		\geq&e^{-t}C(\Psi_1,0,I_+(\Psi_1)_o,f), \ \forall t\in [0,+\infty).
	\end{split}
\end{flalign}
Let $r=e^{-t/2}$, and we get that
\begin{equation}
	\frac{1}{r^2}\int_{\{\Psi_1<2\log r\}}|f|^2\geq C(\Psi_1,0,I_+(\Psi_1)_o,f), \ \forall r\in (0,1].
\end{equation}
Note that $C(\Psi_1,0,I_+(\Psi_1)_o,f)=G(0;\Psi_1,I_+(\Psi_1)_o,f)>0,$
thus Corollary \ref{J-M+f} holds.
\end{proof}

\section{Proof of Corollory \ref{SOPE}}

In this section, we give the proof of Corollary \ref{SOPE}.

\begin{proof}[Proof of Corollary \ref{SOPE}]
Let $\Psi_q:=q\Psi$ for any $q>2a_o^f(\Psi;\varphi_0)\geq 1$. Note that
\[q\Psi=\min\{2q\psi+(2\lceil q \rceil-2q)\log|F|-2\log|F^{\lceil q \rceil}|,0\},\]
where $\lceil q \rceil=\min\{m\in\mathbb{Z} : m\geq q\}$. By the definition of $a_o^f(\Psi;\varphi_0)$, we have $f_o\notin I(2q\Psi+\varphi_0)_o$ for any $q>2a_o^f(\Psi;\varphi_0)$. For any fixed $q>2a_o^f(\Psi;\varphi_0)$, $\lambda>0$, let $\xi\in A^2(\{\Psi<0\},e^{-\varphi_0})^*\setminus\{0\}$, such that $\xi|_{A^2(\{\Psi<0\},e^{-\varphi_0})\cap J_q}\equiv 0$, where $J_q:=I(q\Psi+\varphi_0)_o$. Denote that
\[K_{\xi,q,\lambda}(t):=\sup_{\tilde{f}\in A^2(\{\Psi<0\},e^{-\varphi_0})}\frac{|\xi\cdot \tilde{f}|^2}{\|\tilde{f}\|^2_{q,\lambda,t}},\]
where
\[\|\tilde{f}\|_{q,\lambda,t}:=\left(\int_{\{\Psi<0\}}|\tilde{f}|^2e^{-\varphi_0-\lambda\max\{q\Psi+t,0\}}\right)^{1/2},\]
and $t\in [0,+\infty)$. Theorem \ref{increasing} tells us that $-\log K_{\xi,q,\lambda}(t)+t$ is increasing with respect to $t\in [0,+\infty)$, which implies that
\begin{equation}\label{-logK(t)+t:2}
	-\log K_{\xi,q,\lambda}(t)+t\geq -\log K_{\xi,q,\lambda}(0), \ \forall t\in [0,+\infty).
\end{equation}

Since $\int_{\{\Psi<0\}}|f|^2e^{-\varphi_0}\leq\int_{\{\Psi<0\}}|f|^2e^{-\varphi_0-\Psi}<+\infty$, following from inequality (\ref{-logK(t)+t:2}), we get that
\[\|f\|^2_{q,\lambda,t}\geq\frac{|\xi\cdot f|^2}{K_{\xi,q,\lambda}(t)}\geq e^{-t}\frac{|\xi\cdot f|^2}{K_{\xi,q,\lambda}(0)}, \ \forall t\in [0,+\infty).\]
According to Lemma \ref{B=C}, we have
\begin{flalign}\label{f>e^-tC:2}
	\begin{split}
		\|f\|^2_{q,\lambda,t}&\geq\sup_{\substack{\xi\in A^2(\{\Psi<0\},e^{-\varphi_0})^*\setminus\{0\}\\ \xi|_{A^2(\{\Psi<0\},e^{-\varphi_0})\cap J_q}\equiv 0}}e^{-t}\frac{|\xi\cdot f|^2}{K_{\xi,q,\lambda}(0)}\\
		&=e^{-t}C(q\Psi,\varphi_0,J_q,f), \ \forall t\in [0,+\infty).
	\end{split}
\end{flalign}
Note that for any $t\in [0,+\infty)$,
\begin{equation}\label{f2:2}
	\|f\|^2_{q,\lambda,t}=\int_{\{q\Psi<-t\}}|f|^2e^{-\varphi_0}+\int_{\{0>q\Psi\geq -t\}}|f|^2e^{-\varphi_0-\lambda(q\Psi+t)}.
\end{equation}
Since for any $\lambda>0$,
\[\int_{\{0>q\Psi\geq t\}}|f|^2e^{-\varphi_0-\lambda(q\Psi+t)}\leq\int_{\{0>q\Psi\geq -t\}}|f|^2e^{-\varphi_0}<+\infty,\]
and $\lim_{\lambda\rightarrow+\infty}e^{-\lambda(q\Psi+t)}=0$ on $\{0>q\Psi\geq -t\}$, according to Lebesgue's dominated convergence theorem, we have
\[\lim_{\lambda\rightarrow+\infty}\int_{\{0>q\Psi\geq -t\}}|f|^2e^{-\varphi_0-\lambda(q\Psi+t)}=0.\]
Then equality (\ref{f2:2}) implies
\begin{equation}
	\lim_{\lambda\rightarrow+\infty}\|f\|^2_{q,\lambda,t}=\int_{\{q\Psi<-t\}}|f|^2e^{-\varphi_0}, \ \forall t\in [0,+\infty).
\end{equation}
Thus letting $\lambda\rightarrow+\infty$ in inequality (\ref{f>e^-tC:2}), we get that for any $t\in [0,+\infty)$,
\begin{equation}\label{f>e^-tC2:2}
	\int_{\{q\Psi<-t\}}|f|^2e^{-\varphi_0}\geq e^{-t}C(q\Psi,\varphi_0,J_q,f)=e^{-t}C(\Psi,\varphi_0,J_q,f).
\end{equation}
Note that $J_q\subset I_+(2a_o^f(\Psi;\varphi_0)\Psi+\varphi_0)_o$ for any $q>2a_o^f(\Psi;\varphi_0)$. Then we have
\[C(\Psi,\varphi_0,J_q,f)\geq C(\Psi,\varphi_0,I_+(2a_o^f(\Psi;\varphi_0)\Psi+\varphi_0)_o,f), \ \forall q>2a_o^f(\Psi;\varphi_0).\]
Then it follows from inequality (\ref{f>e^-tC2:2}) that
\begin{equation}\label{f>e^-tC3:2}
	\int_{\{q\Psi<-t\}}|f|^2e^{-\varphi_0}\geq e^{-t}C(\Psi,\varphi_0,I_+(2a_o^f(\Psi;\varphi_0)\Psi+\varphi_0)_o,f)
\end{equation}
for any $q>2a_o^f(\Psi;\varphi_0)$ and $t\in [0,+\infty)$.

According to Fubini's theorem, we have
\begin{flalign*}
	\begin{split}
		&\int_{\{\Psi<0\}}|f|^2e^{-\varphi_0-\Psi}\\
		=&\int_{\{\Psi<0\}}\left(|f|^2e^{-\varphi_0}\int_0^{e^{-\Psi}}\mathrm{d}s\right)\\
		=&\int_0^{+\infty}\left(\int_{\{\Psi<0\}\cap\{s<e^{-\Psi}\}}|f|^2e^{-\varphi_0}\right)\mathrm{d}s\\
		=&\int_{-\infty}^{+\infty}\left(\int_{\{q\Psi<-qt\}\cap\{\Psi<0\}}|f|^2e^{-\varphi_0}\right)e^t\mathrm{d}t.
	\end{split}
\end{flalign*}
Inequality (\ref{f>e^-tC3:2}) implies that for any $q>2a_o^f(\Psi;\varphi_0)$,
\begin{flalign*}
	\begin{split}
		&\int_0^{+\infty}\left(\int_{\{q\Psi<-qt\}\cap\{\Psi<0\}}|f|^2e^{-\varphi_0}\right)e^t\mathrm{d}t\\
		\geq&\int_0^{+\infty}e^{-qt}C(\Psi,\varphi_0,I_+(2a_o^f(\Psi;\varphi_0)\Psi+\varphi_0)_o,f)\cdot e^t\mathrm{d}t\\
		=&\frac{1}{q-1}C(\Psi,\varphi_0,I_+(2a_o^f(\Psi;\varphi_0)\Psi+\varphi_0)_o,f),
	\end{split}
\end{flalign*}
and
\begin{flalign*}
	\begin{split}
		&\int_{-\infty}^0\left(\int_{\{q\Psi<-qt\}\cap\{\Psi<0\}}|f|^2e^{-\varphi_0}\right)e^t\mathrm{d}t\\
		\geq&\int_{-\infty}^0C(\Psi,\varphi_0,I_+(2a_o^f(\Psi;\varphi_0)\Psi+\varphi_0)_o,f)\cdot e^t\mathrm{d}t\\
		=&C(\Psi,\varphi_0,I_+(2a_o^f(\Psi;\varphi_0)\Psi+\varphi_0)_o,f).
	\end{split}
\end{flalign*}
Then we have
\begin{equation}\label{q/q-1}
	\int_{\{\Psi<0\}}|f|^2e^{-\varphi_0-\Psi}\geq \frac{q}{q-1}C(\Psi,\varphi_0,I_+(2a_o^f(\Psi;\varphi_0)\Psi+\varphi_0)_o,f).
\end{equation}
for any $q>2a_o^f(\Psi;\varphi_0)$. Let $q\rightarrow 2a_o^f(\Psi;\varphi_0)+0$, then inequality (\ref{q/q-1}) also holds for $q\geq 2a_o^f(\Psi;\varphi_0)$. Thus if $q>1$ satisfying
\begin{equation}
	\int_{\{\Psi<0\}}|f|^2e^{-\varphi_0-\Psi}< \frac{q}{q-1}C(\Psi,\varphi_0,I_+(2a_o^f(\Psi;\varphi_0)\Psi+\varphi_0)_o,f),
\end{equation}
we have $q<2a_o^f(\Psi;\varphi_0)$, which means that $f_o\in I(q\Psi+\varphi_0)_o$. Proof of Corollary \ref{SOPE} is done.
\end{proof}

\section{Appendix}
\
Let $D$ be a domain in $\mathbb{C}^n$, and $\varphi$ be a plurisubharmonic function on $D$. Denote that
\[\ell_0:=\{\eta=(\eta_{\alpha})_{\alpha\in\mathbb{N}^n} : \exists k\in\mathbb{N}, \text{such\ that\ } \eta_{\alpha}=0, \ \forall |\alpha|\geq k\},\]
where for any $\alpha=(\alpha_1,\ldots,\alpha_n)\in\mathbb{N}^n$, $|\alpha|:=\alpha_1+\cdots+\alpha_n$. Let $z_0\in D$, and $\eta=(\eta_{\alpha})\in\ell_0$. For any $f\in\mathcal{O}(D)$, denote that
\begin{equation}\label{xifunctional}
	\eta\cdot f:=\sum_{\alpha\in\mathbb{N}^n}\eta_{\alpha}\frac{f^{(\alpha)}(z_0)}{\alpha!}.
\end{equation}
It can be shown that for any $\eta\in \ell_0$, there is a finite constant $C_{\eta}>0$, such that
\[|\eta\cdot f|^2\leqslant C_{\eta}\int_D|f|^2e^{-\varphi},\]
for any $f\in A^2(D,e^{-\varphi})$ (see \cite{BG1,BG2}). Then any $\eta\in\ell_0$ can be seen as an element in $A^2(D,e^{-\varphi})^*$ by equality (\ref{xifunctional}). Note that $A^2(D,e^{-\varphi})$ is a Hilbert space. By Riesz representation theorem, there exists $g_{\eta}\in A^2(D,e^{-\varphi})$, such that
\[\eta\cdot f=\int_Df\overline{g_{\eta}}e^{-\varphi}, \ \forall f\in A^2(D,e^{-\varphi}),\]
which induces a map from $\ell_0$ to $A^2(D,e^{-\varphi})$. We denote the map by $T_{\varphi,z_0}$:
\begin{flalign*}
	\begin{split}
		T_{\varphi,z_0} \ : \ \ell_0&\longrightarrow A^2(D,e^{-\varphi})\\
		\eta&\longmapsto g_{\eta}.
	\end{split}
\end{flalign*}
We state the following lemma.

\begin{Lemma}\label{Tvarphi}
	The image of $T_{\varphi,z_0}$ is dense in $A^2(D,e^{-\varphi})$, i.e., $\overline{T_{\varphi,z_0}(\ell_0)}=A^2(D,e^{-\varphi})$, under the topology of $A^2(D,e^{-\varphi})$.
\end{Lemma}

\begin{Remark}
	It follows from Lemma \ref{Tvarphi} that $\ell_0$ is dense in $A^2(D,e^{-\varphi})^*$, under the strong topology of $A^2(D,e^{-\varphi})^*$.
\end{Remark}

\begin{proof}[Proof of Lemma \ref{Tvarphi}]
	We introduce some notations before the proof.
	
	For $\alpha=(\alpha_1,\cdots,\alpha_n), \beta=(\beta_1,\cdots,\beta_n)\in\mathbb{N}^n$, denote that $\alpha<\beta$, if $|\alpha|<|\beta|$, or $|\alpha|=|\beta|$ but there exists $k$ with $1\leq k\leq n$, such that $\alpha_1=\beta_1,\ldots, \alpha_{k-1}=\beta_{k-1}$, $\alpha_k<\beta_k$.

	We may assume that $z_0=o\in D$ is the origin in $\mathbb{C}^n$, and denote $T_{\varphi,z_0}$ by $T$. We will choose a countable sequence $\{\eta[\alpha]\}$ of elements in $\ell_0$, such that
	\[\overline{\text{span}\{T(\eta[\alpha])\}}=A^2(D,e^{-\varphi}),\]
	which can imply Lemma \ref{Tvarphi}. For any $\alpha\in\mathbb{N}^n$, we set $\eta[\alpha]\in\ell_0$, with $\eta[\alpha]_{\gamma}=\gamma!\cdot b_{\gamma}^{\alpha}\in \mathbb{C}$ (which will be determined in the following discussions) for any $\gamma<\alpha$, $\eta[\alpha]_{\alpha}=\alpha!$, and $\eta[\alpha]_{\gamma}=0$ for any $\gamma>\alpha$. Denote that $g_{\alpha}:=g_{\eta_{\alpha}}=T(\eta[\alpha])\in A^2(D,e^{-\varphi})$. We will choose $b_{\gamma}^{\alpha}$ such that
		\begin{flalign}\label{gbeta}
		\begin{split}
			\int_{D}g_{\alpha}\overline{{g}_{\beta}}e^{-\varphi}&=0, \ \forall \alpha\neq\beta;\\
			g_{\beta}^{(\gamma)}(o)&=0, \ \forall \gamma<\beta;\\
			g_{\alpha}^{(\alpha)}(o)&=\int_{D}|g_{\alpha}|^2e^{-\varphi}, \ \forall \alpha.
		\end{split}
	\end{flalign}
	And for any $\alpha\in\mathbb{N}^n$, denote that $\alpha\in S_1$ if $g_{\alpha}\equiv 0$. Otherwise we denote that $\alpha\in S_2$.
		
	Firstly, for $\alpha=(0,\cdots,0)$, we set
	\[\eta((0,\cdots,0))=(1,0,\cdots,0,\cdots)\in\ell_0.\]
	Denote that
	\[T((1,0,\cdots,0,\cdots))=g_{(0,\cdots,0)}\in A^2(D,e^{-\varphi}),\]
	then for any $f\in A^2(D,e^{-\varphi_0})$,
	\begin{equation}\label{f(o)}
		f(o)=\int_D f\overline{{g}_{(0,\cdots,0)}}e^{-\varphi}.
	\end{equation}
	Let $f=g_{(0,\cdots,0)}$ in equality (\ref{f(o)}), we get
	\[g_{(0,\cdots,0)}(o)=\int_D|g_{(0,\cdots,0)}|^2e^{-\varphi}.\]
	
	For some $\alpha\in\mathbb{N}^n$, we assume that for any $\beta<\alpha$, $\eta(\beta)\in \ell_0$ (i.e. the complex number sequence $\{b_{\gamma}^{\beta}\}$) has been choosen to satisfy
		\begin{flalign}\label{gbeta<alpha}
		\begin{split}
			\int_{D}g_{\beta_1}\overline{{g}_{\beta_2}}e^{-\varphi}&=0, \ \forall \beta_1\neq\beta_1, \ \beta_1,\beta_2<\alpha;\\
			g_{\beta}^{(\gamma)}(o)&=0, \ \forall \gamma<\beta;\\
			g_{\beta}^{(\beta)}(o)&=\int_{D}|g_{\beta}|^2e^{-\varphi}, \ \forall \beta<\alpha.
		\end{split}
	\end{flalign}
	 By the choice of $\eta[\alpha]$, for any $f\in A^2(D,e^{-\varphi})$,
	\begin{equation}\label{equation5.1}
		\sum_{\gamma<\alpha}b_{\gamma}^{\alpha}f^{(\gamma)}(o)+f^{(\alpha)}(o)=\int_D f\overline{{g}_{\alpha}}e^{-\varphi}.
	\end{equation}
	Since we want
	\[\int_Dg_{\beta}\overline{g_{\alpha}}e^{-\varphi}=0, \ \forall \beta<\alpha,\]
	then there must be
	\[\sum_{\gamma<\alpha}b_{\gamma}^{\alpha}g_{\beta}^{(\gamma)}(o)+g_{\beta}^{(\alpha)}(o)=0,\  \forall \beta<\alpha,\]
	which is equivalent to
	\begin{equation}\label{lsystem}
		b_{\beta}^{\alpha}g_{\beta}^{(\beta)}(o)+\sum_{\beta<\gamma<\alpha}b_{\gamma}^{\alpha}g_{\beta}^{(\gamma)}(o)=-g_{\beta}^{(\alpha)}(o), \forall \beta<\alpha,
	\end{equation}
	by the choice of $\{g_{\beta}\}$. Equality (\ref{lsystem}) can be seen as a linear equations system for $(b_{\gamma}^{\alpha})_{\gamma<\alpha}$. Note that for $\beta\in S_1$, $g_{\beta}^{(\beta)}=0$. We set $b_{\beta}^{\alpha}=0$ for any $\beta<\alpha$ with $\beta\in S_1$. And we also note that
	\[\prod_{\beta<\alpha,\ \beta\in S_2}g_{\beta}^{(\beta)}(o)=\prod_{\beta<\alpha,\ \beta\in S_2}\int_D|g_{\beta}|^2e^{-\varphi}>0.\]
	It means that there exists $(b_{\gamma}^{\alpha})_{\gamma<\alpha}$ satisfying equality (\ref{lsystem}), where $b_{\beta}^{\alpha}=0$ for any $\beta<\alpha$ with $\beta\in S_1$.
	Now suppose that $(b_{\gamma}^{\alpha})_{\gamma<\alpha}$ is the solution as we described above. Then $g_{\alpha}$ satisfies
	\[\int_D g_{\beta}\overline{{g}_{\alpha}}e^{-\varphi}=0, \ \forall \beta<\alpha.\]
	Note that in the process of induction, for any $\beta<\alpha$, we have
	\[\sum_{\gamma<\beta}b_{\gamma}^{\beta}f^{(\gamma)}(o)+f^{(\beta)}(o)=\int_D f\overline{{g}_{\beta}}e^{-\varphi}.\]
	Let $f=g_{\alpha}$, then we get
	\begin{equation}\label{galphagbeta}
		\sum_{\gamma<\beta}b_{\gamma}^{\beta}g_{\alpha}^{(\gamma)}(o)+g_{\alpha}^{(\beta)}(o)=\int_D g_{\alpha}\overline{{g}_{\beta}}e^{-\varphi}, \ \forall \beta<\alpha.
	\end{equation}
	In equality (\ref{galphagbeta}), by induction, we can know that for any $\beta<\alpha$,
	\begin{equation}\label{galpha0}
		g_{\alpha}^{(\beta)}(o)=\int_D g_{\alpha}\overline{{g}_{\beta}}e^{-\varphi}=\overline{\int_D g_{\beta}\overline{{g}_{\alpha}}e^{-\varphi}}=0.
	\end{equation}
	In addition, in equality (\ref{equation5.1}), Letting $f=g_{\alpha}$, we have
	\[\sum_{\gamma<\alpha}b_{\gamma}^{\alpha}g_{\alpha}^{(\gamma)}(o)+g_{\alpha}^{(\alpha)}(o)=\int_D |g_{\alpha}|^2e^{-\varphi}.\]
	Then it follows from equality (\ref{galpha0}) that
	\[g_{\alpha}^{(\alpha)}(o)=\int_D|g_{\alpha}|^2e^{-\varphi}.\]
	
	Now, by induction, we can choose out $\eta[\alpha]\in\ell_0$ for any $\alpha\in\mathbb{N}^n$ satisfying what we described before equality (\ref{gbeta}), and $\{g_{\alpha}\}_{\alpha\in\mathbb{N}^n}\subset A^2(D,e^{-\varphi})$ satisfies equality (\ref{gbeta}). In the following we prove that
	\begin{equation}
		\overline{\text{span}\{g_{\alpha} : \alpha\in S_2\}}=A^2(D,e^{-\varphi}).
	\end{equation}
	
	For any $f\in A^2(D,e^{-\varphi})$, let $\{a_{\alpha}\}_{\alpha\in\mathbb{N}^n}$ be a sequence of complex numbers (which will be determined in the following). Denote
	\begin{equation}\label{falphaz}
		f_{\alpha}(z)=\sum_{\beta\leq\alpha}a_{\beta}g_{\beta}(z).
	\end{equation}
	We choose $a_{\beta}$ for $\beta\leq\alpha$, such that
	\begin{equation}\label{falpha=f}
		f_{\alpha}^{(\beta)}(o)=f^{(\beta)}(o).
	\end{equation}
	Firstly, for $\beta=(0,\ldots,0)$, if $(0,\ldots,0)\in S_2$, we can see that
	\[a_{(0,\ldots,0)}=\frac{f(o)}{g_{(0,\ldots,0)}(o)}\]
	satisfy equality (\ref{falpha=f}), and if $(0,\ldots,0)\in S_1$, we have $f(o)=0$ according to equality (\ref{f(o)}), which implies $a_{(0,\ldots,0)}=0$ satisfies inequality (\ref{falpha=f}).
	
	Secondly, assume that for some $\gamma\leq\alpha$, all $\beta<\gamma$ have been choosen to satisfy equality (\ref{falpha=f}). According to equality (\ref{falphaz}) and equality (\ref{gbeta}), we have
		\begin{equation}\label{falphagammao}
			f_{\alpha}^{(\gamma)}(o)=\sum_{\beta<\gamma}a_{\beta}g_{\beta}^{(\gamma)}(o)+a_{\gamma}g_{\gamma}^{(\gamma)}(o).
		\end{equation}
	Then
	\begin{equation}\label{leftrightarrow}
		f_{\alpha}^{(\gamma)}(o)=f^{(\gamma)}(o) \Leftrightarrow f^{(\gamma)}(o)=\sum_{\beta<\gamma}a_{\beta}g_{\beta}^{(\gamma)}(o)+a_{\gamma}g_{\gamma}^{(\gamma)}(o).
	\end{equation}
	Note that for $\gamma\in S_2$, $g_{\gamma}^{(\gamma)}(o)=\int_D|g_{\gamma}|^2e^{-\varphi}>0$, then we can choose
	\[a_{\gamma}=(g_{\gamma}^{(\gamma)}(o))^{-1}\left(f^{(\gamma)}(o)-\sum_{\beta<\gamma}a_{\beta}g_{\beta}^{(\gamma)}(o)\right)\]
	to satisfy equality (\ref{falpha=f}). If $\gamma\in S_1$, following from equality (\ref{equation5.1}), we have
	\begin{flalign*}
		\begin{split}
			f^{(\gamma)}(o)&=-\sum_{\beta<\gamma}b_{\beta}^{\gamma}f^{(\beta)}(o)=-\sum_{\beta<\gamma}b_{\beta}^{\gamma}f_{\alpha}^{(\beta)}(o)\\
			&=-\sum_{\beta<\gamma}b_{\beta}^{\gamma}\left(\sum_{\beta'\leq\beta}a_{\beta'}g_{\beta'}^{(\beta)}(o)\right)\\
			&=-\sum_{\beta'<\gamma}a_{\beta'}\left(\sum_{\beta'\leq\beta<\gamma}b_{\beta}^{\gamma}g_{\beta'}^{(\beta)}(o)\right)\\
			&=\sum_{\beta'<\gamma}a_{\beta'}g_{\beta'}^{(\gamma)}(o).
		\end{split}
	\end{flalign*}
	Then we can choose $a_{\gamma}=0$ according to equation (\ref{leftrightarrow}).
	
	Finally, by induction, we can know that the sequence $\{a_{\beta}\}$ can be choosen to satisfy equality (\ref{falpha=f}). In addition, we have $a_{\beta}=0$ for $\beta\in S_1$.
	
	Now we have $(f-f_{\alpha})^{(\beta)}(o)=0$ for any $\beta\leq\alpha$. Then it follows from equality (\ref{equation5.1}) that
	\[\int_D(f-f_{\alpha})\overline{g_{\alpha}}e^{-\varphi}=0,\]
	which means that
	\[\int_D f\overline{g_{\alpha}}e^{-\varphi}=\int_D f_{\alpha}\overline{g_{\alpha}}e^{-\varphi}=a_{\alpha}.\]
	Then combining with equality (\ref{gbeta}), we get that
	\[\sum_{\alpha}|a_{\alpha}|^2\int_D|g_{\alpha}|^2e^{-\varphi}\leq\int_D|f|^2e^{-\varphi}.\]
	Denote
	\[h(z):=\sum_{\alpha\in\mathbb{N}^n}a_{\alpha}g_{\alpha}(z),\]
	then $h\in A^2(\Omega,e^{-\varphi})$. In addition, we have
	\[h(z)=\lim_{|\alpha|\rightarrow +\infty}f_{\alpha}\Rightarrow h^{(\beta)}(o)=\lim_{|\alpha|\rightarrow +\infty}f_{\alpha}^{(\beta)}(o)=f^{(\beta)}(o)\]
	for any $\beta\in\mathbb{N}^n$, inducing that $f\equiv h=\sum_{\alpha\in\mathbb{N}^n}a_{\alpha}g_{\alpha}(z)$. Note that $a_{\alpha}=0$ for $\alpha\in S_1$.
	
	By the arbitrariness of $f\in A^2(D,e^{-\varphi})$, we get
		\begin{equation}
		\overline{\text{span}\{g_{\alpha} : \alpha\in S_2\}}=A^2(D,e^{-\varphi}),
	\end{equation}
	which implies
	\[\overline{\text{span}\{T(\eta[\alpha])\}}=A^2(D,e^{-\varphi}).\]
	Then we know that Lemma \ref{Tvarphi} holds.
\end{proof}

Note that for any $z\in D$, the functional $e_z\in A^2(D,e^{-\varphi})^*$, where
\begin{flalign*}
	\begin{split}
		e_z \ : \ A^2(D,e^{-\varphi})&\longrightarrow\mathbb{C}\\
		f&\longmapsto f(z).
	\end{split}
\end{flalign*}
Then it follows from Riesz representation theorem that there exists $\phi_z\in A^2(D,e^{-\varphi})$ such that
\[e_z\cdot f=\int_Df\overline{\phi_z}e^{-\varphi}, \ \forall f\in A^2(D,e^{-\varphi}).\]
According to the following Lemma, we can also prove Lemma \ref{xihol}.

\begin{Lemma}\label{spanphiz}
	Under the topology of $A^2(D,e^{-\varphi})$, we have
	\[\overline{\text{span}\{\phi_z : z\in D\}}=A^2(D,e^{-\varphi}).\]
\end{Lemma}

\begin{proof}
	Denote that
	\[H:=\overline{\text{span}\{\phi_z : z\in D\}},\]
	then $H$ is a closed subspace of $A^2(D,e^{-\varphi})$. If $H\subsetneq A^2(D,e^{-\varphi})$, there exists $h\in A^2(D,e^{-\varphi})$ such that $h\neq 0$, and
	\[\int_D h\overline{\phi_z}e^{-\varphi}=0,\ \forall z\in D.\]
	However, we have $h(z)=e_z\cdot h=\int_D h\overline{\phi_z}e^{-\varphi}=0$ for any $z\in D$, inducing that $h\equiv 0$, which is a contradiction. It means that
	\[\overline{\text{span}\{\phi_z : z\in D\}}=A^2(D,e^{-\varphi}).\]
\end{proof}

\begin{Remark}
	It follows from Lemma \ref{spanphiz} that $\text{span}\{e_z : z\in D\}$ is dense in $A^2(D,e^{-\varphi})^*$, under the strong topology. And in Lemma \ref{xihol}, for any $\eta\in \text{span}\{e_z : z\in D\}$ such that
	\[\eta=\sum_{k=1}^Nc_k e_{z_k},\]
	where $N$ is a finite positive integer, $c_k\in\mathbb{C}$, and $z_k\in D$ for any $k$, we have that
	\[\eta\cdot f_{\tau}=\sum_{k=1}^Nc_ke_{z_k}\cdot f_{\tau}=\sum_{k=1}^Nc_kf(\tau,z_k)\]
	is holomorphic with respect to $\tau$. Then with a similar discussion in the proof of Lemma \ref{xihol}, we can know that Lemma \ref{xihol} can also be induced by Lemma \ref{spanphiz}.
\end{Remark}

\vspace{.1in} {\em Acknowledgements}. We would like to thank Zhitong Mi and Zheng Yuan for checking this paper. The second named author was supported by National Key R\&D Program of China 2021YFA1003103, NSFC-11825101, NSFC-11522101 and NSFC-11431013.

\bibliographystyle{references}
\bibliography{xbib}

\end{document}